\documentclass[12pt]{article}
\usepackage{graphicx}
\usepackage{epsfig}
\usepackage{fancyhdr}
\usepackage{setspace}

\usepackage{amsmath,amsfonts,amssymb,amsthm}
\theoremstyle{plain}
\newtheorem{thm}{Theorem}[section]
\newtheorem{lem}{Lemma}[section]
\newtheorem{prop}{Proposition}[section]

\theoremstyle{definition}

\newcommand{\reftit}{\textit}    
\newcommand{\refis}{\textbf}     
\newcommand{\Rea}{\mathbb{R}}

\numberwithin{equation}{section}

\begin{document}

\author{Yevgeniy Kovchegov\footnote{Department of Mathematics, Oregon State University, Kidder Hall, Corvallis, OR 97331; Email:
kovchegy@math.oregonstate.edu; Phone: 541-737-1379; Fax: 541-737-0517.} \hspace{3.0cm} Ne\c{s}e Y{\i}ld{\i}z\footnote{Corresponding author: Department of Economics,
University of Rochester, 231 Harkness Hall, Rochester, NY
14627; Email:
nese.yildiz@rochester.edu; Phone: 585-275-5782; Fax: 585-256-2309.}}
\title{Orthogonal Polynomials for Seminonparametric Instrumental Variables Model\thanks{Many of the results of this paper were presented as part of a larger project at University of Chicago and Cowles Foundation, Yale University, econometrics research seminar in the spring of 2010, as well as the 2010 World Congress of the Econometric Society in Shanghai. We would like to thank participants of those seminars for valuable comments and questions. We would also like to thank the editors and an anonymous referee for valuable comments. This work was partially supported by a grant from the Simons Foundation ($\#$284262 to Yevgeniy Kovchegov).}}
\date{}
\maketitle

\begin{abstract}

We develop an approach that resolves a {\it polynomial basis problem} for a class of models with discrete endogenous covariate, and for a class of econometric models considered in the work of Newey and Powell \cite{np}, where  the endogenous covariate is continuous.  Suppose $X$ is a $d$-dimensional endogenous random variable, $Z_1$ and $Z_2$ are the instrumental variables (vectors), and $Z=\left(\begin{array}{c}Z_1 \\Z_2\end{array}\right)$. Now, assume that the conditional distributions of $X$ given $Z$ satisfy the conditions sufficient for solving the identification problem as in Newey and Powell \cite{np} or as in Proposition \ref{discr_density} of the current paper. That is, for a function $\pi(z)$ in the image space there is a.s. a unique function $g(x,z_1)$ in the domain space such that
$$E[g(X,Z_1)~|~Z]=\pi(Z) \qquad Z-a.s.$$
In this paper, for a class of conditional distributions $X|Z$, we produce an orthogonal polynomial basis $\{Q_j(x,z_1)\}_{j=0,1,\hdots}$ such that for a.e.  $Z_1=z_1$, and for all $j \in \mathbb{Z}_+^d$, and a certain $\mu(Z)$,
$$P_j(\mu(Z))=E[Q_j(X, Z_1)~|~Z ],$$  where $P_j$ is a polynomial of degree $j$.
This is what we call solving the {\it polynomial basis problem}.

Assuming the knowledge of $X|Z$ and an inference of $\pi(z)$, our approach provides a natural way of estimating the structural function of interest $g(x,z_1)$. Our polynomial basis approach is naturally extended to Pearson-like and Ord-like families of distributions.



\noindent
\end{abstract}

\section{Introduction}\label{intro}
In this paper we start with a small step of extending the set of econometric models for which nonparametric or semiparametric identification of structural functions is guaranteed to hold by showing completeness  when the endogenous covariate is discrete with unbounded support. Note that the case of discrete endogenous covariate $X$ with unbounded support is not covered by the sufficiency condition given in \cite{np}. Then, using the theory of differential equations we develop a novel orthogonal polynomial basis approach for a large class of the distributions given in Theorem 2.2 in \cite{np}, and in the case of discrete endogenous covariate $X$ for which the identification problem is solved in this paper. Our approach is new in economics and provides a natural link between identification and estimation of structural functions. We also discuss how our polynomial basis results can be extended to the case when the conditional distribution of $X|Z$ belongs to either the modified
Pearson or modified Ord family.

Experimental data are hard to find in many social sciences. As a result, social scientists
often have to devise statistical methods to recover causal effects of variables (covariates) on outcomes of interest. When the structural relationship between a dependent variable and the explanatory variables (i.e. $g(x,z_1)$) is parametrically specified Instrumental variables (IV) method is typically
used to get consistent and asymptotically normal estimators for the finite dimensional
vector of parameters, and thus, the structural function of interest.\footnote{A keyword search for ``instrumental variables'' in JSTOR returned more than 20,000 entries.} However the parametric
estimators are not robust to misspecification of the underlying structural relationship, $g(x, z_1)$. For
example, in the context of the analysis of consumer behavior recent empirical studies have
suggested the need to allow for a more flexible role for the total budget variable to capture the
observed consumer behavior at the microeconomic level. (See \cite{bck} and the references therein.) Failure of 
robustness of parametric methods raises the question whether it is possible to extend
the instrumental variables estimation to non-parametric framework. This question was
first studied in \cite{np}. Thus far,
however, the development of theoretical analysis and empirical implementation of nonparametric
instrumental variables methods have been slow. This may have to do with the fact that
identification is very hard to attain in these models. In addition, although there are some
results about convergence rates of nonparametric estimators of the structural function,
or on asymptotic distribution of the structural function evaluated at finitely many values
of covariates\footnote{See \cite{hhor, chenreiss, chenpouzo, cgs, horlee}.} to date the asymptotic distribution of the estimator for the structural function
is still unknown.

In this paper we suggest a semiparametric approach. This suggestion is motivated by the fact that sufficient conditions for nonparametric identification are closely related to the conditional distribution of the endogenous covariate given the instruments, which can be estimated non-parametrically since it only depends on observable quantities. We suggest a way of nonparametrically estimating the structural function while assuming that the conditional distribution of the endogenous covariate given instruments belongs to a large family for which identification of the structural function is guaranteed to hold. Ours is not the first paper which suggests taking a related semiparametric approach to attack this problem. \cite{hh1} and \cite{bck} both take a semiparametric approach in analyzing 
the Engel curve relationship. The semiparametric approach in \cite{hh1} is different from the one taken by \cite{bck}, and is more closely related to the one taken in this paper. In particular, \cite{bck} assume $g(X,Z_1)=h(X-\phi(Z_1^T\theta_1))+Z_1^T\theta_2$, with $\theta_1, \theta_2$ as finite dimensional parameters, $\phi$ having a known functional form, and $h$ non-parametric, but leave the distribution of $X$ given $Z$ to be more flexible than in \cite{hh1}. In contrast, \cite{hh1} leave specification of $g$ more flexible, but assume that the joint distribution of $X$ and $Z_2$ conditional on $Z_1$ is normal. 

The Engel curve relationship describes the expansion path for commodity demands as the household’s budget increases. In Engel curve analysis $Y$ denotes budget share of the household spent on a subgroup of goods, $X$ denotes log total expenditure allocated by the household to the subgroup of goods of interest, $Z_{1}$ are variables describing other observed characteristics of households, and $U$ represents unobserved heterogeneity across households. The (log) total expenditure variable, $X$, is a choice variable in the household’s allocation of income across consumption goods and savings. Thus, household's optimization suggests that $X$
is jointly determined with household's demands for particular goods and is, therefore, likely to be 
an endogenous regressor, or a regressor that is related to $U$, in the estimation of Engel curves. This means that the conditional mean of $Y$ estimated by nonparametric least squares regression
cannot be used to estimate the economically meaningful structural Engel curve relationship. Fortunately, as argued in \cite {bck}, household's
allocation model does suggest exogenous sources of income that will provide suitable instrumental
variables for total expenditure in the Engel curve regression. In particular, log disposable household income is believed to be exogenous because the driving unobservables like
ability are assumed to be independent of the preference orderings which play an important role in household's allocation decision and are included in $U$ (see \cite{hh1}). Consequently, log disposable income is usually taken as the excluded instrument, $Z_2$.
\cite{hh1} demonstrates that  log expenditure and log disposable income variables are both well characterized by
joint normality, conditional on other variables describing household characteristics. Under the assumption that the joint distribution of $X$ and $Z_2$ conditional on $Z_1$ is normal \cite{hh1} provide a semiparametric estimator for the structural Engel curve and give convergence rates for their estimator. In parametric models normality is typically
associated with nice behavior, but in a nonparametric
regression with endogenous regressors the situation is very different. Indeed, it is well established that joint normality can lead
to very slow rates of convergence (see \cite{bck, dfr, st}). In contrast to \cite{hh1} we suggest an estimation method that is directly related to the information contained in the identification condition and that covers any conditional distribution of $X$ given $Z$ (not just normal distribution) that belongs to a large family for which identification of the structural function is known to hold. By exploiting this information our method eliminates one step of estimation. As a result, we expect estimators that are based on our method will have a faster rate of convergence.  Specifically, the case where the joint distribution of $X$ and $Z_2$ conditional on $Z_1$ is normal as in \cite{hh1} fits right into the orthogonal polynomial framework of this paper. This correspondence will be pointed out in a remark in Subsection \ref{res_cts}. The follow-up paper that includes a least square analysis for normal conditional distributions is being prepared by the authors.

Our approach to choosing the orthogonal polynomials for approximating structural function is semiparametric and is motivated by the form of the conditional density (either with respect to Lebesgue or counting measure) of covariates given instruments. Using the form of this density function we can derive a second-order Stein operator (called Stein-Markov operator in \cite{schoutens}) whose eigenfunctions are orthogonal polynomials (in covariates) under certain sufficient conditions. This step utilizes the generator approach from Stein's theory originated in Barbour \cite{barbour90} and extensively studied in Schoutens \cite{schoutens}. One could use the eigenfunctions of the Stein-Markov operator to approximate the structural functions of interest in such models. Since the conditional expectations of these orthogonal basis functions given instruments are known up to a certain function of the instruments (namely, they are polynomials in $\mu(Z)$, which will be defined below), this approach is likely to simplify estimation. The in-depth  information on Stein's method and Stein operators can be found in \cite{barbour, barbour90, shao, schoutens, stein} and references therein.

A common way of estimating the structural function, which depends on the endogenous regressor $X$, starts with picking a basis, $\{Q_j\}_{j=1}^\infty$, for the space the structural function of interest belongs to. Finitely many elements of this basis is used to approximate the structural function. To estimate the coefficients on the elements of the basis, both the left hand side, or dependent variable, and the finite linear combination of the basis functions are first projected on the space defined by the instrument $Z$, and then the projection of the dependent variable is regressed onto the linear combination of the projections of basis functions. When this is done, typically, the choice of basis functions has little to do with the conditional distribution of $X|Z$, and hence, with the conditions that ensure identification of the structural function. As a result, the projections of the basis functions on the instrument are not known analytically, but have to be estimated by
 non-parametric regression. In this paper, we propose a method that links the condition for identification of the structural function to the choice of the basis used to approximate this function in estimation stage. We do this by exploiting the form of the conditional density of covariates given instruments. As suggested above we propose the use of the eigenfunctions of the Stein-Markov operator to approximate the structural function. Since the conditional expectations of these orthogonal basis functions given instruments are known up to a certain function of the instruments, this would eliminate one step of the estimation of the structural function. It should be stressed, however, even assuming the conditional density of covariates given instruments is known up to finite dimensional parameters, does not imply that the conditional expectations of arbitrary basis functions given instruments are necessarily known analytically.




The paper is organized as follows. Subsection \ref{id-discr} discusses the identification result for the case of discrete endogenous covariate $X$ with unbounded support. Section \ref{poly-base} contains the orthogonal polynomial approach for the basis problem. Finally, Section \ref{concl} contains the concluding remarks.

\subsection{An identification result} \label{id-discr}

As it will be shown in Subsection \ref{sub:discrete}, our approach to choosing orthogonal basis works for many cases in which the endogenous variable is discrete and has unbounded support. To be able to talk about such cases we state an identification result that covers those cases. This theorem as well as Theorem 2.2 of \cite{np} follow from Theorem 1 on p.132 of \cite{leh}. We let $X$ denote the endogenous random variable and $Z=\left(\begin{array}{c}Z_1 \\Z_2\end{array}\right)$ denote the vector of instrumental variables.
\begin{prop}\label{discr_density}
Let $X$ be a random variable, with conditional density (w.r.t. either Lebesgue or counting measure) of $X|Z$ given by
\begin{equation*}
p(x|Z=z):=p(x|z)=t(z)s(x,z_1)\prod\limits_{j=1}^d [\mu_j(z)-m_j]^{\tau_j(x,z_1)} \qquad \tau(x,z_1) \in \mathbb{Z}_+^d,
\end{equation*}
where $t(z)>0$, $s(x,z_1)>0$, $\tau(x,z_1)=(\tau_1(x,z_1),\dots,\tau_d(x,z_1))$ is one-to-one in $x$, and the support of $\mu(Z)=(\mu_1(Z),\dots,\mu_d(Z))$ given $Z_1$ contains a non-trivial open set in $\mathbb{R}^d$, and $\mu_j(Z)>m_j~~(Z-a.s.)$ for each $j=1,\dots,d$.
Then
 \begin{equation*}
 E[g(X,Z_1)|Z_1,Z_2]=0 \quad Z-a.s. \quad \text{ implies } \quad g(X,Z_1)=0 \quad (X,Z_1)-a.s.
 \end{equation*}
 \end{prop}
\begin{proof}\footnote{For the case in which $X$ is discrete an alternative proof can be found in \cite{ky}.} Note that
\[
p(x|z) = t(z)s(x,z_1)\exp{\left[\sum_{i=1}^d\tau_i(x,z_1)\log{(\mu_i(z)-m_i)}\right]}.
\]
Then letting $A(\eta)=0$, and $\eta_i=\log{(\mu_i(z)-m_i)}$, we see that the result follows from \cite{lc}. See also \cite{leh}.
\end{proof}
The above theorem extends Theorem 2.2 in \cite{np}, where it was shown that if with probability one conditional on $Z$, the distribution of $X$ is absolutely continuous w.r.t. Lebesgue measure, and its conditional density is given by
\begin{equation}\label{npdensity}
f_{X|Z}(x|z)= t(z)s(x,z_1) \exp{\left[\mu(z) \cdot \tau(x,z_1)\right]},
\end{equation}
where $t(z)>0$, $s(x,z_1)>0$, $\tau(x,z_1)$ is one-to-one in $x$, and the support of $\mu(Z)$ given $Z_1$ contains a non-trivial open set, then for each $g(x,z_1)$ with finite expectation \mbox{$E[g(X,Z_1)|Z]=0\;\;$} $(Z-a.s.)$ implies that $g(X,Z_1)=0\;\;$ $(X,Z_1)-a.s$.

The condition requiring the support of $\mu(Z)$ given $Z_1$ to contain a nontrivial open set in $\Rea^d$ in both our Theorem \ref{discr_density} and Theorem 2.2 in \cite{np} can be weakened to requiring that the support of $\mu(Z)$ given $Z_1$ be a countable set that is dense in a nontrivial open set in $\Rea^d$.

\section{Polynomial basis results} \label{poly-base}

Once again, let $X$ be a $d$-dimensional endogenous random variable, $Z_1$ and $Z_2$ be the instrumental variables (vectors), and $Z=\left(\begin{array}{c}Z_1 \\Z_2\end{array}\right)$. Now, assume that the conditional distributions of $X$ given $Z$ satisfy the conditions sufficient for solving the identification problem as in Theorem 2.2 of \cite{np}  or as in Proposition \ref{discr_density} of the current paper. Then, for a function $\pi(z)$ in the image space there is a unique function $g(x,z_1)$ in the domain space such that
$$E[g(X,Z_1)~|~Z]=\pi(Z) \qquad Z~a.s.$$
In this section we will use Stein-Markov operators to solve the {\it polynomial basis problem} for a class of conditional distributions $X|Z$. Specifically, we will develop an approach to finding an orthogonal polynomial basis $\{Q_j(x,z_1)\}_{j=0,1,\hdots}$ such that for a.e.  $Z_1=z_1$, and for all $j \in \mathbb{Z}_+^d$, and a function $\mu(Z)$ defined in Section \ref{intro},
$$P_j(\mu(Z))=E[Q_j(X, Z_1)~|~Z ],$$  where $P_j$ is a polynomial of degree $j$.
See \cite{barbour, shao, schoutens, stein} for comprehensive studies and reviews of Stein-Markov operators and Stein's method.
In the examples with no instrumental variable $Z_1$, i.e. $Z=Z_2$, polynomials $Q_j(x,z_1)$ will be denoted by $Q_j(x)$.

\subsection{Sturm-Liouville Equations and Stein operators}
 Let open set $\Omega(z) \in \Rea^d$ be the support of $X$ given $Z=z$, and let $\partial \Omega(z)$ denote the boundary of $\Omega(z)$. Consider a continuous conditional density function $f_{X|Z}(x|z)=s(x,z_1)t(z) e^{\mu(z)^T  \tau(x,z_1)}$ as in Theorem 2.2 in \cite{np} with $x=(x_1,\dots,x_{d})^T$ and $\mu(z)=\big(\mu_1(z),\dots,\mu_{d}(z)\big)^T$ in $\Rea^d$, and $t(z)>0$. Assume that for $a.e.\; Z_1=z_1$,  $~\tau(x,z_1)=\big(\tau_1(x,z_1),\hdots,\tau_d(x,z_1)\big)^T$ is a twice differentiable invertible one-to-one function from  $\Omega(z) \subseteq \Rea^d$ to $\Rea^d$ with nonzero partial derivatives, and $s(x,z_1):\Rea^d \rightarrow \Rea$ is a differentiable function in $x$. 
Next denote by $\nabla_{x,\tau}$ the following first order linear operator 
$$\nabla_{x,\tau} f(x):=\left({\partial \over \partial x_1} \left[{f(x)\over {\partial \tau_1(x,z_1) \over \partial x_1}}\right],\dots , {\partial \over \partial x_d}\left[{f(x)\over {\partial \tau_d(x,z_1) \over \partial x_d}}\right]\right)$$ 

\noindent
 We differentiate  $f_{X|Z}(x|z)$ to obtain
 $$\nabla_{x,\tau} f_{X|Z}(x|z) ={\nabla_{x,\tau} s(x,z_1) \over s(x,z_1)}f_{X|Z}(x|z)+\mu(Z)^T f_{X|Z}(x|z) \qquad \text{ for all } x \in \Omega(z).$$
 The following statement holds for almost every $Z=z$. For a function $Q(x,z_1)$ that is differentiable in $x$ and satisfies $~Q(x,z_1) s(x,z_1) \Big / {\partial \tau_i(x,z_1) \over \partial x_i}=0~$ for each $i$ and each $x\in \partial \Omega(z)$,\footnote{If $\partial \Omega(z)$ contains a singularity or a point at infinity, this statement should be taken to hold in the limit.} we integrate by parts to obtain
 \begin{equation}\label{main}
  E[AQ(X,Z_1)|Z]=-\mu(Z)^TE[Q(X,Z_1)|Z]  \qquad Z ~a.s.,
 \end{equation}
 where
  \begin{equation}\label{eqn:stop}
  AQ(x,z_1)=\frac{1}{s(x,z_1)}\nabla_{x,\tau} [s(x,z_1)Q(x,z_1)]={\big(\nabla_{x,\tau} s(x,z_1) \big) Q(x,z_1)\over s(x,z_1)}+\sum_{i=1}^d { ~\frac{\partial Q(x,z_1)}{\partial x_i}~ \over {\partial \tau_i(x,z_1) \over \partial x_i}}.
   \end{equation}

 \vskip 0.2 in
 \noindent
 Now, for a given $z$, let $L^2(\Rea^d,s(x,z_1))$ denote the space of Lebesgue measurable $u(x,z_1)$ in $x$ such that $\int\limits_{\Omega(z)} u^2(x,z_1)s(x,z_1)\mathrm{d}x<\infty$,  with the inner product
 $$\big< u,v\big>_s:=\int\limits_{\Omega(z)} u(x,z_1) v(x,z_1)s(x,z_1)\mathrm{d}x.$$
 
 \noindent
 Next define the following Sturm-Liouville operator:
$${\cal A}Q :=  {1\over s(x,z_1)}\nabla_{x,\tau} \Big[s(x,z_1)\nabla_x Q(x,z_1))\Big] =  {\nabla_{x,\tau}  s(x,z_1) \cdot \nabla_x Q(x, z_1)\over s(x,z_1)}+\sum_{i=1}^d {1 \over {\partial \tau_i(x,z_1) \over \partial x_i}} \frac{\partial^2 Q(x,z_1)}{\partial x_i^2},$$
where $~\nabla_x :=\left({\partial \over \partial x_1},\dots , {\partial \over \partial x_d} \right)^T$ is standard gradient.
 Here $A$ is a Stein operator for the distribution that has Lebesgue density equal to ${s(x,z_1)\over \int s(x,z_1)\mathrm{d}x}$, and $\mathcal{A}$ is the corresponding Stein-Markov operator.
 
  \vskip 0.2 in
 \noindent
 Then, integration  by parts shows $\cal{A}$ is a {\bf self-adjoint operator} with respect to $\big<\cdot , \cdot \big>_s$. Specifically, $~\big< {\cal A} u,v\big>_s=\big< u,{\cal A}v\big>_s$ provided the following standard boundary conditions
 \begin{equation} \label{eqn:dO}
 \sum_{i=1}^d \int\limits_{\partial \Omega(z)} \left[\left( {\partial \over \partial x_i} u(x,z_1) \right) v(x,z_1) -\left( {\partial \over \partial x_i} v(x,z_1) \right) u(x,z_1)\right]{s(x,z_1) \over {\partial \tau_i(x,z_1) \over \partial x_i}} ~d\Gamma(x)=0
 \end{equation}
 $Z~a.s.$ for all $u(x,z_1)$ and $v(x,z_1)$ in $\mathcal{C}^2(\Rea^d) \cap L^2(\Rea^d,s(x,z_1))$ for almost every  $~Z_1=z_1$. 
 Trivially, the above boundary conditions (\ref{eqn:dO}) are satisfied if
  \begin{equation} \label{eqn:dOt}
  \left( {\partial \over \partial x_i} u(x,z_1) \right) v(x,z_1) -\left( {\partial \over \partial x_i} v(x,z_1) \right) u(x,z_1) \equiv 0 \quad \text{ on }~\partial \Omega(z).
  \end{equation} 
  In the case of a singularity or a point at infinity on the boundary the above boundary conditions (\ref{eqn:dOt}) will need to hold in the limit. The eigenvalues $\lambda_j$ of $\cal{A}$ are all real, and the corresponding eigenfunctions $Q_j(x,z_1)$ solve the following Sturm-Liouville differential equation
 \begin{equation}\label{diffl_eqn1}
 \sum_{i=1}^d {s(x,z_1) \over {\partial \tau_i(x,z_1) \over \partial x_i}} \frac{\partial^2 Q_j(x,z_1)}{\partial x_i^2}+\sum_{i=1}^d \frac{\partial }{\partial x_i} \!\! \left({s(x,z_1) \over {\partial \tau_i(x,z_1) \over \partial x_i}}\right) \frac{\partial Q_j(x,z_1)}{\partial x_i}-\lambda_j s(x,z_1) Q_j(x,z_1)=0.
 \end{equation}
These $~Q_j(x,z_1)~$ form  a basis of $L^2(\Rea^d,s(x,z_1))$, orthogonal with respect to $\big<\cdot , \cdot \big>_s$.

\subsubsection{A special case}
 Assume that  for $a.e. ~Z_1=z_1$, $s(x,z_1)\in C^\infty (\Rea^d)$ w.r.t. variable $x$,
 for each nonnegative integer $j=(j_1,\dots,j_d)$. Consider a special case when
$~Q_j(x,z_1)={(-1)^{j_1+\dots+j_d} \over s(x,z_1)} {\partial^{j_1+\dots+j_d} \over \partial x^{j_1}_1 \dots \partial x^{j_d}_d} s(x,z_1)~$
are the orthogonal eigenfunctions in $L^2(\mathbb{R}^d,s(x,z_1))$, then  their projections
$$P_j(Z):=E[Q_j(X, Z_1) |Z]=\prod_{k=1}^d\mu_k(Z)^{j_k}=\mu(Z)^j$$
due to integration by parts under the boundary conditions requiring the corresponding boundary integral to be zero.

 \vskip 0.2 in
 \noindent
{\bf Example:} In particular, using the Rodrigues' formula for the Sturm-Liouville boundary value problem, we can show that when $$s(x,z_1)=\gamma(z_1)\exp{\left[\alpha(z_1)\frac{x^Tx}{2}+\beta(z_1)\right]},$$ with $\alpha(z_1)<0$ for each $z_1$, there is a series of eigenvalues $\lambda_0,\lambda_1, \lambda_2,...$ that lead to solutions $\{Q_j(x,z_1)\}_{j=0}^\infty$, where each $~Q_j(x,z_1)={(-1)^{j_1+\dots+j_d} \over s(x,z_1)} {\partial^{j_1+\dots+j_d} \over \partial x^{j_1}_1 \dots \partial x^{j_d}_d} s(x,z_1)~$ is a multidimensional Hermite-type orthogonal polynomial basis for $L^2(\mathbb{R}^d,s(x,z_1))$.\footnote{When $s(x,z_1)$ is of this form $Q_j(x,z_1)$ are polynomials. In general equation (\ref{diffl_eqn1}) may have solutions for other $s(x,z_1)$ that are not necessarily polynomials.}

\subsection{The orthogonal polynomial basis results for continuous $X$}\label{res_cts}

We assume that $d=1$ in this subsection with the exception of Example 2 below. Then
$$\frac{\partial f_{X|Z}(x|z)}{\partial x} = \frac{\frac{\partial s(x,z_1)}{\partial x}}{s(x)}f_{X|Z}(x|z)+\mu(z)\frac{\partial \tau(x,z_1)}{\partial x}f_{X|Z}(x|z).$$
and 
$$
AQ(x,z_1) = \frac{\partial}{\partial x}\left(\frac{s(x,z_1)Q(x,z_1)}{\frac{\partial \tau(x,z_1)}{\partial x}}\right)\frac{1}{s(x,z_1)}
=
\frac{\frac{\partial Q(x,z_1)}{\partial x}}{\frac{\partial \tau(x,z_1)}{\partial x}}
+
\frac{\frac{\partial s(x,z_1)}{\partial x}}{s(x,z_1)}
\frac{Q(x,z_1)}{\frac{\partial \tau(x,z_1)}{\partial x}}
-
\frac{Q(x,z_1)\frac{\partial^2 \tau(x,z_1)}{\partial x^2}}{\left[\frac{\partial \tau(x,z_1)}{\partial x}\right]^2}
$$
as in (\ref{eqn:stop}). Once again, equation (\ref{main}) is satisfied  if $~Q(x,z_1)s(x,z_1) \Big/ \frac{\partial \tau(x,z_1)}{\partial x}=0~$ on $\partial \Omega(z)$ for $a.e.\; Z=z$.
here, for $d=1$, Stein-Markov operator is
$$
\mathcal{A}Q(x,z_1) :=A{\partial Q(x,z_1) \over \partial x}= \frac{\frac{\partial^2 Q(x,z_1)}{\partial x^2}}{\frac{\partial \tau(x,z_1)}{\partial x}}
+
\left(\frac{\frac{\partial s(x,z_1)}{\partial x}}{s(x,z_1)}
\frac{1}{\frac{\partial \tau(x,z_1)}{\partial x}}
-
\frac{\frac{\partial^2 \tau(x,z_1)}{\partial x^2}}{\left[\frac{\partial \tau(x,z_1)}{\partial x}\right]^2}
\right)\frac{\partial Q(x,z_1)}{\partial x}.
$$
We would like to find eigenfunctions $Q_j$ and eigenvalues $\lambda_j$  of ${\cal A}$ such that $~\mathcal{A}Q_j=\lambda_jQ_j$.
We define 
$$\phi(x,z_1):=-\frac{1}{\frac{\partial \tau(x,z_1)}{\partial x}} \quad \text{ and } \quad \psi(x,z_1):=-\frac{1}{\frac{\partial \tau(x,z_1)}{\partial x}}\left[\frac{\frac{\partial s(x,z_1)}{\partial x}}{s(x,z_1)}-\frac{\frac{\partial^2 \tau(x,z_1)}{\partial x^2}}{\frac{\partial \tau(x,z_1)}{\partial x}}\right],$$ 
Then Sturm-Liouville differential  equation (\ref{diffl_eqn1}) can be rewritten as
\begin{equation}\label{diffl_eqn}
\phi(x,z_1)\frac{\partial^2 Q(x,z_1)}{\partial x^2}+\psi(x,z_1)\frac{\partial Q(x,z_1)}{\partial x}+\lambda Q(x,z_1)=0.
\end{equation}
with the boundary conditions (\ref{eqn:dOt}) rewritten as
\begin{eqnarray}\label{boundry}
c_1Q(\alpha_1(z_1),z_1)+c_2\frac{\partial Q(\alpha_1(z_1),z_1)}{\partial x} = 0 & & c_1^2+c_2^2>0,\\
d_1Q(\alpha_2(z_1),z_1)+d_2\frac{\partial Q(\alpha_2(z_1),z_1)}{\partial x} = 0 & & d_1^2+d_2^2>0, \nonumber
\end{eqnarray}
where $~\Omega(z)=\big( \alpha_1(z_1), \alpha_2(z_1) \big)$ denotes the support of $X$ conditioned on $Z_1=z_1$. 
The solution to this Sturm-Liouville type problem exists when one of the three sufficient conditions listed below is satisfied. See \cite{szego} and \cite{schoutens}.\footnote{\cite{schoutens} and \cite{szego} give results for Hermite, Laguerre and Jacobi polynomials, the other cases are obtained by defining $\tilde{x}=a x+b$ and applying the results in \cite{schoutens} and \cite{szego}. Also note that these conditions are sufficient for the solutions to be polynomials. Solutions that are not polynomials, but nevertheless form an orthogonal basis might exist under less restrictive conditions.} Moreover, in the cases we list below, the solutions are {\bf orthogonal polynomials} with respect to the weight function $~s(x,z_1)$,
and for each $j$, the corresponding eigenfunction $Q_j(x,z_1)$ is proportional to
$$\frac{1}{s(x,z_1)}\frac{\partial^j }{\partial x^j}\left(s(x,z_1)[\phi(x,z_1)]^j\right).$$
Here $Q_0$ is a constant eigenfunction corresponding to $\lambda_0=0$. 
Finally, iterating equation (\ref{main}) proves the following important result.
\begin{thm}\label{thm:main}
Suppose $Q_j(x,z_1)$ are an orthogonal polynomial basis $Z~a.s.$ Then functions $P_j(Z)=E[Q_j(X,Z_1)|Z]$ are $j^{th}$ order {\bf polynomials} in $\mu(Z)$ with its coefficients being functions of $Z_1$.
\end{thm}
\begin{proof}
Observe that $P_0 \equiv Q_0$ is a constant. Consider $j>0$, since $f_{X|Z}(x|z)$ satisfies the unique identification condition stated in Theorem 2.2 of \cite{np} (that in turn is a Corollary of Theorem 1 of \cite{leh}), 
$E[\mathcal{A}Q_j(X,Z_1)|Z] =\lambda_j E[Q_j(X,Z_1)|Z]\not=0$. Therefore $\lambda_j \not=0$, and since $~\mathcal{A}Q_j=\lambda_jQ_j$, 
$$P_j(Z)=E[Q_j(X,Z_1)|Z]={1 \over \lambda_j}E[\mathcal{A}Q_j(X,Z_1)|Z] ={1 \over \lambda_j}E\left[A{\partial \over \partial x}Q_j(X,Z_1) \Big| Z\right] ,$$
where $~{\partial \over \partial x}Q_j(x,z_1)=\sum\limits_{i=0}^{j-1} a_i Q_i(x,z_1)$ is a polynomial of degree $j-1$ in $x$. Therefore
$$P_j(Z)={a_0P_0 \over \lambda_j}+\sum\limits_{i=1}^{j-1} {a_i \over \lambda_j} E[AQ_i(X,Z_1)~|Z]={a_0P_0 \over \lambda_j}-\mu(Z)\sum\limits_{i=1}^{j-1} {a_i \over \lambda_j} P_i(Z)$$
by  (\ref{main}). The statement of the theorem follows by induction.
\end{proof}

 \vskip 0.2 in
 \noindent
Next we list the sufficient conditions for the eigenfunctions $\{Q_j(x,z_1)\}_{j=0}^\infty$ to be orthogonal polynomials in $x$ that form a basis in $L^2(\mathbb{R}^d,s(x,z_1))$, together with the corresponding examples of continuous conditional densities $f_{X|Z}(x|z)$. 

\begin{enumerate}
\item \textbf{Hermite-like polynomials:} $\phi$ is a non-zero constant, $\psi$ is linear and the leading term of $\psi$ has the opposite sign of $\phi$. In this case, let $\phi(x,z_1)=c(z_1)\neq 0$, then $\tau(x,z_1)=-\frac{1}{c(z_1)}x+d(z_1)$. Then, $\psi(x,z_1)=c(z_1)\frac{\frac{\partial s(x,z_1)}{\partial x}}{s(x,z_1)}=a(z_1)x+b(z_1)$. Thus, we have $\frac{\frac{\partial s(x,z_1)}{\partial x}}{s(x,z_1)}=\frac{a(z_1)}{c(z_1)}x+\frac{b(z_1)}{c(z_1)}$. Let $\alpha(z_1):=a(z_1)/c(z_1)$ and $\beta(z_1):=b(z_1)/c(z_1)$, where $\alpha(z_1)< 0$ $\forall z_1$, since $a(z_1)$ and $c(z_1)$ always have opposite signs. Solving for $s(x,z_1)$ we get $s(x,z_1)=\gamma(z_1)\exp{\big(\alpha(z_1)x^2/2+\beta(z_1)x\big)}$. 
 \vskip 0.2 in
 \noindent
\textbf{Example 1:} Given a function $\sigma(z_1) \not=0$, and suppose $d=1$. Consider
$$f_{X|Z}(x|z) = \frac{1}{\sqrt{2\pi\sigma^2(z_1)}}\exp{\left\{-\frac{(x-\tilde{\mu}(z))^2}{2\sigma^2(z_1)}\right\} }.$$ 
Then $~~t(z)=\frac{1}{\sqrt{2\pi\sigma^2(z_1)}}\exp{\left\{-{z_2^2 \over 2\sigma^2(z_1)}\right\} }$, $~~s(x,z_1)=\exp{\left\{-{x^2 \over 2\sigma^2(z_1)}\right\} }$, $~~\mu(z)=\tilde{\mu}(z)/\sigma^2(z_1)$, and ${\tau(x,z_1)=x}$. The orthogonal polynomials $Q_j(x,z_1)$ are
 $$Q_j(x,z_1)=(-1)^j e^{{x^2\over 2\sigma^2(z_1)}}\frac{\mathrm{d}^j}{\mathrm{d}x^j}e^{-{x^2\over 2\sigma^2(z_1)}},$$
$P_j(z)={\tilde{\mu}(z)^j\over \sigma^{2j}(z_1)}=\big[\mu(z) \big]^j$ and $\lambda_j=-j$ for each $j>1$. 
 \vskip 0.2 in
 \noindent
\textbf{Remark:} In \cite{hh1} it is assumed that 
\begin{align*}
\left(\begin{array}{c} X \\ Z_2\end{array}\right)|Z_1=z_1 ~~\sim N\left(\left(\begin{array}{c}\mu_X(z_1)\\ \mu_{Z_2}(z_1)\end{array}\right), \left[\begin{array}{cc}
\sigma^2_X(z_1) & \sigma_{XZ_2}(z_1)
\\
\sigma_{XZ_2}(z_1) & \sigma^2_{Z_2}(z_1)
\end{array}\right]\right).
\end{align*}
This corresponds to Example 1 above with 
$$\tilde{\mu}(z_1,z_2)=\mu_{X}(z_1)+\frac{\sigma_{XZ_2}(z_1)}{\sigma^2_X(z_1)}(z_2-\mu_{Z_2}(z_1))$$
and
$$\sigma^2(z_1) = \left[1-\frac{\sigma^2_{XZ_2}(z_1)}{\sigma^2_X(z_1)\sigma^2_{Z_2}(z_1)}
\right]\sigma^2_X(z_1).$$

\vskip 0.2 in
 \noindent
\textbf{Example 2:} Suppose $d>1$.
 For $x=(x_1,\dots,x_d)^T$ and $z_2=(z'_1,\dots,z'_d)^T$, let
 $f_{X|Z}(x|z)={\sqrt{\det M} \over (2\pi)^{d \over 2}}e^{-{(x-z_2)^T M (x-z_2) \over 2}}$, where $M=M(z_1)$ is the inverse of the variance-covariance $d \times d$ matrix function with $~\det M(z_1) >0$.
 Then $~t(z)={\sqrt{\det M} \over (2\pi)^{d \over 2}}e^{-{z^T Mz \over 2}}$, $~s(x,z_1)=e^{-{x^TMx\over 2}}$, $\mu(z)=M z_2$, and $\tau(x,z_1)=x$.  For each nonnegative integer-valued $j=(j_1,\dots,j_d)$, the orthogonal polynomial $Q_j(x,z_1)$ is given by
 $$Q_j(x,z_1)=(-1)^{j_1+\dots+j_d} e^{x^T Mx \over 2} {\partial^{j_1+\dots+j_d} \over \partial^{j}x^{j_1}_1 \dots \partial x^{j_d}_d} e^{-{x^T Mx \over 2}}.$$
Then
 $$P_j(Z)=E[Q_j(X) |Z]~=(e_1^T M Z_2)^{j_1}\dots (e_d^T M Z_2)^{j_d}~=\big(e_1 \cdot \mu(Z)\big)^{j_1}\dots\big(e_d \cdot \mu(Z)\big)^{j_d}=\big[\mu(z) \big]^j,$$
where $e_1,\hdots,e_d$ denote standard basis vectors, and for any vector $w=(w_1,w_2,\hdots, w_d)^T$, $~w^j:=w_1^{j_1}w_2^{j_2}\hdots w_d^{j_d}$ .

\item \textbf{Laguerre-like polynomials:} $\phi$ and $\psi$ are both linear, the roots of $\phi$ and $\psi$ are different, and the leading terms of $\phi$ and $\psi$ have the same sign if the root of $\psi$ is less than the root of $\phi$ or vice versa.\\
Suppose $\phi(x,z_1)=a(z_1)x+b(z_1)$ and $\psi(x,z_1)=c(z_1)x+d(z_1)$ with $b(z_1)/a(z_1)\neq d(z_1)/c(z_1)$. Then
\[
\frac{\partial \tau(x,z_1)}{\partial x}=\frac{1}{-a(z_1)x-b(z_1)},
\]
so
\[
  \tau(x,z_1)=\frac{1}{a(z_1)}\log[{a(z_1)x+b(z_1)|}+C(z_1).
  \]
 Moreover,
\[
 \psi(x,z_1)=[a(z_1)x+b(z_1)]\frac{\frac{\partial s(x,z_1)}{\partial x}}{s(x,z_1)}+a(z_1)=c(z_1)x+d(z_1)
 \Leftrightarrow
  \frac{\frac{\partial s(x,z_1)}{\partial x}}{s(x,z_1)}=\frac{c(z_1)x+d^*(z_1)}{a(z_1)x+b(z_1)},
 \]
where $d^*(z_1)=d(z_1)-a(z_1)$. This means that
 \[
 s(x,z_1)=\rho(z_1)\exp{\left\{\int \frac{c(z_1)x+d^*(z_1)}{a(z_1)x+b(z_1)}\mathrm{d}x\right\}}.
 \]

 \vskip 0.2 in
 \noindent \textbf{Example:} Suppose $d=1$. Let $\delta, r >0$ and a function $g: \mathbb{R} \rightarrow \mathbb{R}$ be given, and let $\Gamma(\cdot )$ denote the gamma function. Consider
 \[
 f_{X|Z}(x|z)={1 \over \Gamma(r+z_2)}\delta^{r+z_2} \big(x-g(z_1)\big)^{r+z_2-1}e^{-\delta (x-g(z_1))}~~\text{ for }~x > g(z_1),
 \] where $Z_2>-r$. Then $~~t(z) = {1 \over \Gamma(r+z_2)}\delta^{r+z_2}$, $~~s(x,z_1)= \big(x-g(z_1)\big)^{r-1}e^{-\delta (x-g(z_1))}$, $\mu(z)=z_2$, and $~~\tau(x,z_1)=\log{\big(x-g(z_1)\big)}$, since $ \big(x-g(z_1)\big)^{z_2}=e^{z_2\log{(x-g(z_1))}}$. In this case, $\phi(x,z_1)=- \big(x-g(z_1)\big)$ and $\psi(x,z_1) = \delta  \big(x-g(z_1)\big) -r$. The orthogonal polynomials $Q_j(x,z_1)$ are
 $$Q_j(x,z_1)={{ \big(x-g(z_1)\big)^{-(r-1)}e^{\delta  (x-g(z_1))}}\over j!} ~\frac{\mathrm{d}^j}{\mathrm{d}x^j}\left[ \big(x-g(z_1)\big)^{j+r-1}e^{-\delta (x-g(z_1))}\right],$$
for $j>1$, $~P_j(z)=z_2(z_2-1)\cdots (z_2-n+1)$, and $~\lambda_j = -\delta j$.

\item \textbf{Jacobi-like polynomials:} $\phi$ is quadratic, $\psi$ is linear, $\phi$ has two distinct real roots, the root of $\psi$ lies between the two roots of $\phi$, and the leading terms of $\phi$ and $\psi$ have the same sign.

In this case,
\[
\frac{\partial \tau(x,z_1)}{\partial x}=-\frac{1}{(x-r_1(z_1))(x-r_2(z_1))},
\]
with $r_1\neq r_2$ and $x$ not equal to either one of them. In this case, however, $\tau$ is not one-to-one on $x$, and the condition given in Theorem 2.2 of \cite{np} does not hold unless specific support conditions are met.

Solving the last differential equation we get
\[
\tau(x,z_1)=\frac{1}{r_1(z_1)-r_2(z_1)}\left[\log{|x-r_2(z_1)|}-\log{|x-r_1(z_1)|}\right]+c(z_1).
\]
Plugging this into the formula for $\psi$ yields
\[
\psi(x,z_1)=(x-r_1(z_1))(x-r_2(z_1))\left[\frac{\frac{\partial s(x,z_1)}{\partial x}}{s(x,z_1)}
+\frac{2x-r_1(z_1)-r_2(z_1)}{(x-r_1(z_1))(x-r_2(z_1))}\right]=a(z_1)x+b(z_1).
\]
Rearranging terms gives us
\begin{eqnarray*}
\frac{\frac{\partial s(x,z_1)}{\partial x}}{s(x,z_1)}
& = &
-\frac{2x-r_1(z_1)-r_2(z_1)}{(x-r_1(z_1))(x-r_2(z_1))}\\
& + &
\frac{1}{r_1(z_1)-r_2(z_1)}\left[\frac{a(z_1)r_1(z_1)+b(z_1)}{x-r_1(z_1)}-\frac{a(z_1)r_2(z_1)+b(z_1)}{x-r_2(z_1)}\right]\\
& =: & \kappa(x,z_1).
\end{eqnarray*}
Let $\alpha(x,z_1):=\int \kappa(x,z_1)\mathrm{d}x$. Then
\begin{eqnarray*}
\alpha(x,z_1) & = & -\log{|(x-r_1(z_1))(x-r_2(z_2))|}\\
& + & \frac{a(z_1)r_1(z_1)+b(z_1)}{r_1(z_1)-r_2(z_1)}\log{|x-r_1(z_1)|}\\
& - & \frac{a(z_1)r_2(z_1)+b(z_1)}{r_1(z_1)-r_2(z_1)}\log{|x-r_2(z_1)|},
\end{eqnarray*}
and
\[
s(x,z_1) = \rho(z_1)\exp{[\alpha(x,z_1)]}.
\]

 \vskip 0.2 in
 \noindent \textbf{Example:}  Suppose for simplicity that there is no $Z_1$ (so that $z=z_2$), and 
$$f_{X|Z}(x|z)={1 \over {\cal B}(a+z,b-z)}x^{a+z-1}(1-x)^{b-z-1}~~\text{ for }x\in (0,1),$$ 
where $\mathcal{B}(\cdot,\cdot)$ denotes the beta function. Suppose the following condition is satisfied:
 \begin{equation}\label{Jacobi_boundary}
\lim\limits_{x \rightarrow 0+} x^{a+Z} Q(x)=\lim\limits_{x \rightarrow 1-} (1-x)^{b-Z} Q(x)=0 \qquad Z-a.s.
\end{equation}
 We also assume the support of $Z$ is in $(-a,b)$. Then $~\mu(z)=z$, $~t(z)={{\cal B}(a,b) \over {\cal B}(a+z,b-z)}$, and $~s(x)={1\over {\cal B}(a,b)}x^{a-1}(1-x)^{b-1}$. Finally, $\tau(x)=\log{\left({x\over 1-x}\right)}$ since $\left({x\over 1-x}\right)^z=\exp{\left[z\log{\left({x\over 1-x}\right)}\right]}$. Then $\phi(x)=-x(1-x)$ and
 $\psi(x)=(a-b)x-a$. The orthogonal polynomial $Q_j$ are the scaled Jacobi polynomials and satisfy the following hypergeometric differential equations of Gauss:
$$x(1-x)Q_j''+(a -(a+b)x)Q_j'+j(j+a+b-1)Q_j=0$$
for each degree $j=0,1,\dots$. See section 4.21 of \cite{szego}, and \cite{ww}. These scaled Jacobi polynomials can be expressed with the hypergeometric functions
 $$Q_j(x) :=P_j^{(a-1,b-1)}(1-2x)={(\alpha)_j \over j!} \cdot ~ _2F_1(-j,j+a+b-1;a; x) ~,$$
 where $(\alpha)_j:=\alpha (\alpha+1)\cdots (\alpha+j-1)$, and for $c \notin\mathbb{Z}_-$, $_2F_1(a,b;c; x):=\sum_{j=0}^\infty \frac{(a)_j(b)_jx^j}{(c)_j j!}$.
 Note that these $Q_j$'s satisfy equation (\ref{Jacobi_boundary}). Moreover, the eigenvalues are $~\lambda_j=-j(j+a+b-1)$ and for $j>1$,
 $$P_j(Z)=E[Q_j(X) |Z]=-{Z \over \lambda_j} E[Q_j'(X)|Z].$$

\end{enumerate}

\subsection{The orthogonal polynomial basis results for  discrete $X$}\label{sub:discrete}

Here we show that the orthogonal polynomial basis results of the previous section go through when $X$ is discrete and satisfies the conditions in Theorem \ref{discr_density}. Suppose for simplicity $X$ is one-dimensional with its conditional distribution given by
\begin{equation}\label{discr_density1}
P(X=x|Z=z):=p(x|z)=t(z)s(x,z_1)[\mu(z)-m]^x
\end{equation}
for $$x\in a+\mathbb{Z}_+=\{a, a+1, a+2,\dots\},$$ where $\mu(Z)>m\; a.s.$, and a given $-\infty \le a< \infty$.

\vskip 0.2 in
\noindent
For a function $h$, define respectively the backwards and forwards difference operators as
\begin{eqnarray*}
\nabla h(x)& := & h(x)-h(x-1), \\
\Delta h(x) & := & h(x+1)-h(x).
\end{eqnarray*}
Let $A h(x,z_1) :=  \frac{s(x-1,z_1)}{s(x,z_1)}\nabla h(x,z_1)-\left[m+\frac{s(x-1,z_1)}{s(x,z_1)}\right]h(x,z_1)$, and let $s(a-1,z_1)=0$ for almost every $Z=z$.
\begin{lem}
Suppose $g$ is such that $E[g(X,Z_1)]<\infty$. Then  $$E[Ag(X,Z_1)|Z]=-\mu(Z)E[g(X,Z_1)|Z] \qquad (Z-a.s.)$$
\end{lem}
\begin{proof}
\begin{eqnarray*}
E[Ag(X,Z_1)|Z] & = & \sum_{x\in a+\mathbb{Z}_+}\frac{s(x-1,Z_1)}{s(x,Z_1)}[g(x,Z_1)-g(x-1,Z_1)]t(z)s(x,Z_1)[\mu(Z)-m]^x\\
& - & \sum_{x\in a+\mathbb{Z}_+} \left[m+\frac{s(x-1,Z_1)}{s(x,Z_1)}\right]g(x,Z_1)t(z)s(x,Z_1)[\mu(Z)-m]^x
\\
& = &
[m-\mu(Z)]  \sum_{x\in a+\mathbb{Z}_+} g(x-1,Z_1)t(z)s(x-1,Z_1)[\mu(Z)-m]^{x-1}\\
& - & m \sum_{x\in a+\mathbb{Z}_+} g(x,Z_1)t(z)s(x,Z_1)[\mu(Z)-m]^{x}=-\mu(Z)E[g(X,Z_1|Z].
\end{eqnarray*}
\end{proof}
Note that the result holds when the support of $p(x|z)=P(x=x|Z=z)$ is $$a-\mathbb{Z}_+=\{\dots,a-2, a-1, a\}$$ with $-\infty<a <\infty,~$ $~A h(x,z_1) :=  \frac{s(x+1,z_1)}{s(x,z_1)}\Delta h(x,z_1)-\left[m+\frac{s(x+1,z_1)}{s(x,z_1)}\right]h(x,z_1),~$ and
\mbox{$s(a+1,z_1)=0$} for almost every $Z=z$.

\vskip 0.2 in
\noindent
From the above lemma we see that equation (\ref{main}) holds, and iterating on that equation yields
\begin{equation}\label{iki}
 E[A^k g(X)|Z]=(-\mu(Z))^kE[g(X)|Z].
 \end{equation}
The corresponding Stein-Markov operator ${\cal A}$ is defined as ${\cal A}h=A\Delta h$.  The eigenfunctions of ${\cal A}$ are orthogonal polynomials $Q_j$ such that
 $${\cal A}Q_j(x,z_1)=\lambda_j Q_j(x,z_1).$$
See \cite{szego}, \cite{schoutens}.  Then by (\ref{main}) and (\ref{iki}) we have
\[
\lambda_jE[Q_j(X, Z_1)|Z]=E[A\Delta Q_j (X)|Z] = -\mu(Z) E[\Delta Q_j(X, Z_1)|Z],
\]
so that
$$E[Q_j(X, Z_1) |Z]={-\mu(Z) \over \lambda_j} E[\Delta Q_j(X, Z_1)|Z]$$ for $j>1$.  Thus, we know recursively that $P_j(Z):=E[Q_j(X, Z_1) |Z]$ is a $j$-th degree polynomial in $\mu(Z)$, as in Theorem \ref{thm:main} of the preceding subsection.

\vskip 0.2 in
\noindent
We now present the following specific examples.
\begin{enumerate}
    \item \textbf{Charlier polynomials:} Suppose there is no $Z_1$, and $X|Z$ has a Poisson distribution with density $p(x|z)= {e^{-(\tilde{m}_0+z)}[\tilde{m}_0+z]^x\over x!}=e^{-z}{e^{-\tilde m_0}\tilde m_0^x\over x!}\left[1+{z\over \tilde m_0}\right]^x$, for $x\in \mathbb{N}$, so that $t(z)=e^{-z}$, $s(x)={e^{-\tilde m_0}\tilde m_0^x\over x!}$, $m_0=1$, and $\mu(z)={z\over \tilde m_o}$. Then $Ah(x)=h(x)-{x\over \tilde m_0}h(x-1)$ is the Stein operator. The eigenfunctions of the Stein-Markov operator are the Charlier polynomials
 $Q_j(x)=C_j(x;\tilde m_0)(x)=\sum_{r=0}^j \binom{j}{r} (-1)^{j-r}\tilde m_0^{-r}x(x-1)\dots(x-r+1)$
 which are orthogonal w.r.t. Poisson-Charlier weight measure $\rho(x):={e^{-\tilde m_0} \tilde m_0^x \over x!}\sum_{k=0}^{\infty}\delta_k(x)$, where $\delta_k(x)$ equals 1 if $k=x$, and 0 otherwise. See \cite{schoutens}.
 Finally,\\
 $P_j(Z)=E[Q_j(X) |Z]=\sum_{r=0}^j \sum_{x=r}^{\infty} e^{-(\tilde m_0+Z)}{(\tilde m_0+Z)^x \over (x-r)!} \binom{j}{r} (-1)^{j-r}\tilde m_0^{-r}={Z^j \over \tilde m_0^j}.$
    \item \textbf{Meixner polynomials:}  Suppose there is no $Z_1$, and for $x\in \mathbb{N}$ and $\alpha$ an integer greater than or equal to 1, $p(x|z)=\binom{x+\alpha -1}{x}
    p^\alpha[1-p+\mu(z)]^xt(z)$,  where
    $t(z)=\left[\sum_{x=0}^\infty{\Gamma(x+\alpha)\over x! \Gamma(\alpha)}p^\alpha [1-p+\mu(z)]^x\right]^{-1}$. The above lemma applies with
     $s(x)=\binom{x+\alpha -1}{x}p^\alpha$, $m_0=1-p$. Then $Ah(x)= (1-p) h(x)-{x\over x+\alpha}h(x-1)$ is the Stein operator. The eigenfunctions of the Stein-Markov operator are the Meixner polynomials
 $Q_j(x)=M_j(x;\alpha,p)(x)=\sum_{k=0}^j (-1)^{k}\binom{j}{k}\binom{x}{k}k!(x-\alpha)_{j-k}p^{-k}$, where $(a)_j:=a(a+1)\hdots (a+j-1)$.
 which are orthogonal w.r.t. weight measure $\rho(x):=s(x)\sum_{k=0}^{\infty}\delta_k(x)$. 
\end{enumerate}

\subsection{Extension to Pearson-like and Ord-like Families}

Suppose there is no $Z_1$, i.e. $Z=Z_2$.  Suppose $\phi(x)$ is a polynomial of degree at most two and $\psi(x)$ is a decreasing linear function on an interval $(a,b)$. Also $\phi(x)>0$ for $a<x<b$,  $\phi(a)=0$ if $a$ is finite, and $\phi(b)=0$ if $b$ is finite.
If $\xi$ is a random variable with either Lebesgue density or density with respect to counting measure $f(x)$ on $(a,b)$ that satisfies
\begin{equation} \label{eqn:ord}
D[\phi(x) f(x)] = \psi(x)f(x),
\end{equation}
where $D$ denotes derivative when $\xi$ is continuous, and the forward difference operator $\Delta$ when $\xi$ is discrete. Then the above relation (\ref{eqn:ord}) describes the {\bf Pearson family} when $\xi$ is continuous and {\bf Ord family}, when $\xi$ is discrete. Many continuous distributions fall into the Pearson family, and many discrete ones fall into Ord's family. See \cite{schoutens} and the references therein.

\vskip 0.2 in
\noindent
Suppose $\xi$ is a random variable in either Pearson or Ord family. Following \cite{schoutens}, define its Stein operator as 
$$AQ(x) = \phi(x) D^*Q(x)+\psi(x) Q(x)$$
for all $Q$ such that $~E[Q(\xi)]< \infty~$ and $~E[D^*Q(\xi)]<\infty$, where $D^*$ denotes the derivative when $\xi$ is continuous and the backwards difference operator $\nabla$ when $\xi$ is discrete. Then $E[AQ(\xi)]=0$.
Let the corresponding Stein-Markov operator, $\cal{A}$, be defined as $\mathcal{A} Q:= ADQ$.

\vskip 0.2 in
\noindent
Now, consider a Stein operator $AQ(x) = \phi(x) D^*Q(x)+\psi(x) Q(x)$ together with the corresponding Stein-Markov operator $\cal{A}$ for some random variable in either Pearson or Ord family.  Let $Q_j$ be the orthogonal polynomial eigenfunctions of $\cal{A}$. Consider random variables $X$ and $Z$, where the conditional distribution of $X$ given $Z$ is such that the Stein operator of $X$ given $Z$ equals
$$A_\mu Q=\phi D^*Q+(\psi+c\mu(Z))Q,$$
where $c$ is a constant. Then $E[A_\mu Q(X)|Z]=0$. Now, since $Q_j$ are eigenfunctions of $\cal{A}$,
\begin{eqnarray*}
\lambda_jE[Q_j(X)|Z] & = & E[\mathcal{A}Q_j(X)|Z]=E[ADQ_j(X)|Z]=E[(A-A_\mu)DQ_j(X)|Z]
\\
& = & -c\mu(Z)E[DQ_j(X)|Z].
\end{eqnarray*}
Letting $P_j(Z):=E[Q_j(X)|Z]$ we see that $P_j$'s are $j^{th}$-order polynomials in $\mu(Z)$ as $DQ_j(x)$ can be expressed as a linear combination of $Q_0(x),Q_1(x),\hdots,Q_{j-1}(x)$ in the above equation analogous to (\ref{main}). Thus our main result Theorem \ref{thm:main} applies whenever the Stein operator of $X|Z$ is expressed as $A_\mu Q=\phi D^*Q+(\psi+c\mu(Z))Q$. The question then arises for which, if any, conditional distributions of $X|Z$ the Stein operator is of this form. It should be pointed out that this current approach extends to multidimensional discrete $X|Z$, and other types of distributions with well defined Stein operators. We now give some examples for such discrete distributions.\\

\noindent \textbf{Examples:}
\begin{enumerate}
%

    \item {\bf Binomial distribution:} It is known that
    $$AQ(x) = (1-p)x\nabla Q(x)+[pN-x]Q(x)$$
    is the Stein operator for a Binomial random variable with parameters $N$ and $p$. In this case, $\phi(x) = (1-p)x$ and $\psi(x)=pN-x$. See \cite{schoutens}.

    Suppose $X|Z\sim Bin (N+\mu(Z),p)$, with $\mu(Z)\in \mathbb{Z}_+$. Then
    \[
    A_\mu Q(x) = (1-p)x\nabla Q(x)+[pN+p\mu(Z)-x]Q(x)
    \]
 Let $Q_{-1}(x):= 0$, $Q_1(x)=0$, and \mbox{$Q_j(x)=K_j(x,N,p)=\sum_{l=0}^j (-1)^{j-l}\binom{N-x}{j-l}\binom{x}{l}p^{j-l}(1-p)^l$}, the Krawtchouk polynomials, are orthogonal with respect to the binomial $Bin(N,p)$ distribution.

    \item {\bf Pascal / Negative binomial distribution:}  It is known that
    $$AQ(x) = x\nabla Q(x)+[(1-p)\alpha-px]Q(x)$$
    is the Stein operator for a negative binomial random variable with parameters $\alpha$ and $p$. In this case, $\phi(x) = x$ and $\psi(x)=(1-p)\alpha-px$. See \cite{schoutens}.

    Suppose
$$P(X=x|Z=z)=p(x|z)=\binom{x+\alpha+\mu(z)-1}{x}p^{\alpha+\mu(z)}(1-p)^x,$$
for $x\in \mathbb{N}_+$. Then
\[
A_\mu Q(x) = x\nabla Q(x)+[(1-p)\alpha+(1-p)\mu(Z)-px]Q(x)
\]
In this case, $Q_j=M_j(x;\alpha,p)$, where $M_j(x;\alpha,p)$ denote Meixner polynomials which were defined in the previous section and are orthogonal with respect to the Pascal distribution with parameter vector $(\alpha,p)$.

\end{enumerate}

\section{Conclusion} \label{concl}

In this paper we introduced an identification problem for nonparametric and semiparametric models in the case when the conditional distribution of $X$ given $Z$ belongs to the generalized power series distributions family.\footnote{We borrow this term from \cite{jkk}} Using an approach based on differential equations, Sturm-Liouville theory specifically, we solved orthogonal polynomial basis problem for the conditional expectation transformation, $E[g(X)|Z]$.
Finally, we discussed how our polynomial basis results can be extended to the case when the conditional distribution of $X|Z$ belongs to either the modified Pearson or modified Ord family.

In deriving our results we encountered a second order differential (or difference, in the case of discrete $X$) equation with boundary values, which is a Sturm-Luiouville type equation. In this paper we focused on cases in which the solutions to the Sturm-Liuouville problem, which are the eigenfunctions of the operator $\mathcal{A}$, are an orthogonal polynomial basis.
Our approach is more general than this. In particular, one might question for what conditional distributions the eigenfunctions of the Stein-Markov operator $\cal{A}$ are orthogonal basis functions, but not necessarily orthogonal polynomials. Our paper does not address this question. Addressing this question is left for future research. Finally, the work of applying the orthogonal polynomial basis approach for estimating structural functions is nearing completion.


\bibliographystyle{amsplain}

\begin{thebibliography}{99}  

\bibitem{barbour}
   Barbour, A. D. and Chen, L.H.Y. (2005): \reftit{An introduction to Stein's method}, Singapore University Press
   
\bibitem{barbour90}
   Barbour, A. D. (1990): \reftit{Stein's method for diffusion approximations}, Probability Theory and Related Fields \refis{84} Vol. 3,  297-322.

\bibitem{bck}
    Blundell, R., X. Chen, and D. Kristensen (2007): \reftit{Semi-Nonparametric IV Estimation of Shape-Invariant Engel Curves}, Econometrica, 75, 1613-1669.

\bibitem{shao}
   Chen, L.H.Y., Goldstein, L., and Shao, Q.M (2011): \reftit{Normal approximation by Stein's method}, Springer

\bibitem{chenpouzo}
	Chen X. and D. Pouzo (2012): \reftit{Estimation of Nonparametric Conditional Moment Models With Possibly Nonsmooth Generalized Residuals}, Econometrica, 80, 277-321.

\bibitem{chenreiss}
	Chen X. and M. Reiss (2011): \reftit{On Rate Optimality for Ill-posed Inverse Problems in Econometrics}, Econometric Theory, 27, 497-521.

\bibitem{cgs}
	Chernozhukov, V., P. Gagliardini and O. Scaillet (2008): \reftit{Nonparametric Instrumental  Variable
Estimation of Quantile Structural Effects}, Working Paper, HEC University of Geneva and Swiss Finance Institute.

\bibitem{dfr}
    Darolles, S., J. P. Florens, and E. Renault (2006): \reftit{Nonparametric Instrumental Regression}, Econometrica, 79, 1541-1565.

\bibitem{hhor}
    Hall, P. and J.L. Horowitz (2005): \reftit{Nonparametric methods for inference in the presence of instrumental
variables}, Annals of Statistics 33, 2904-2929.

\bibitem{hh1}
    Hoderlein, S. and H. Holzmann (2011): \reftit{Demand analysis as an ill-posed problem with semiparametric specification},  Econometric Theory, 27, 460-471.


\bibitem{hormander}
        H\"ormander, L. (1973):
       \reftit{An Introduction to Complex Analysis in Several Variables (second ed.)},
       North-Holland Mathematical Library, Vol. \refis{7}.
      
\bibitem{horlee}
	Horowitz, J. L. and S. Lee (2012): \reftit{Uniform Confidence Bands for Functions Estimated Nonparametrically with Instrumental Variables}, Journal of Econometrics, 168, 175-188.

\bibitem{jkk}
    Johnson, N. L., S. Kotz and A. W. Kemp (1992): \reftit{Univariate Discrete Distributions(second ed.)}, Wiley Series in Probability and Statistics

\bibitem{ky}
    Kovchegov, Y. V. and N. Y{\i}ld{\i}z (2011):  \reftit{Identification via completeness for discrete covariates and orthogonal polynomials}, Oregon State University Technical Report.

\bibitem{leh}
    Lehmann, E. L. (1959): \reftit{Testing Statistical Hypotheses}, Wiley, New York

\bibitem{lc}
    Lehmann, E. L., S. Fienberg (Contributor) and G. Casella (1998): \reftit{Theory of Point Estimation}, Springer Texts in Statistics

\bibitem{np}
	Newey, W. K. and J. L. Powell (2003):  \reftit{Instrumental Variable Estimation of Nonparametric Models}, Econometrica, 71, 1565-1578.

\bibitem{schoutens}
        Schoutens, W. (2000):
       \reftit{Stochastic Processes and Orthogonal Polynomials},
        Lecture notes in statistics (Springer-Verlag), Vol. \refis{146}.
        
\bibitem{st}
    Severini, T.A. and G. Tripathi (2006): \reftit{Some identification issues in nonparametric linear models with
endogenous regressors}, Econometric Theory 22, 258-278.

\bibitem{stein}
   Stein, C. (1986): \reftit{Approximate computation of expectations}, Institute of Mathematical Statistics Lecture Notes, Monograph Series


\bibitem{szego}
        Szeg\"{o}, G. (1975):
       \reftit{Orthogonal Polynomials (fourth ed.)}, AMS Colloquium Publications, Vol. \refis{23}.

\bibitem{ww}
        Whittaker, E. T. and G. N. Watson (1935):
        \reftit{A Course of Modern Analysis (fourth ed.)}, Cambridge Mathematical Library.

\end{thebibliography}

\end{document}